\documentclass[11pt]{article}
\usepackage{amsfonts,latexsym,rawfonts,amsmath,amssymb,amsthm,graphicx}
\numberwithin{equation}{section}
\newtheorem{theorem}{Theorem}[section]

\newtheorem{thm}[theorem]{Theorem}
\newtheorem{pro}[theorem]{Proposition}
\newtheorem{cor}[theorem]{Corollary}

\title{New inequalities for eigenvalues of the Dirichlet Laplacian on the hyperbolic space}
\author{Yong Luo
\\ \\In Commemoration of the 80th Birthday of Professor Weiyue Ding}
\date{}
\begin{document}
	\maketitle
	\begin{abstract} In this paper, motivated by study on universal inequalities for eigenvalues of the Dirichlet Laplacian, we prove some new inequalities for eigenvalues of the Dirichlet Laplacian on the hyperbolic space. In particular, we verify Cheng's conjecture (Adv. Lect. Math. 37, 2017) up to loss of $\epsilon$ for two special kinds of bounded domains in the hyperbolic space.
	\end{abstract}
	\section{Introduction}
Let $(M^n, g)$  be an $n$-dimensional complete Riemannian manifold and $\Omega$  a bounded connected domain in $M$. Let  $\Delta$ denote the Laplacian  acting on functions on $ M $. Eigenvalues of the Dirichlet Laplacian is stated as:
     	\begin{eqnarray}\label{Dirichlet problem}
     	\left\{\begin{array}{c}
     		\Delta u=-\lambda u \text { in } \Omega, \\
     		u|_{\partial \Omega}=0 .
     	\end{array}\right.
     \end{eqnarray}
     \indent
     Let
     \begin{eqnarray*}
     	0<\lambda_{1}<\lambda_{2} \leq \lambda_{3} \leq \cdots
     \end{eqnarray*}
     denote the successive eigenvalues of (\ref{Dirichlet problem}). Here each eigenvalue is repeated according to its multiplicity. It is well-known that  Weyl's asymptotic formula holds (cf. \cite{Cha}):
\begin{eqnarray}\label{Weyl asy.}
\lambda_k\sim\frac{4\pi^2}{(\omega_n Vol\Omega)^\frac{2}{n}}k^\frac{2}{n}, \ k\to \infty,
\end{eqnarray}
where $\omega_n$ is the volume of the unit ball in $\mathbb{R}^n$.\\
     \indent
Generally, if an inequality of eigenvalues requires no hypotheses on the geometric quantities of the domain $\Omega$ (other than its dimension), it is called an universal inequality. Universal inequalities constitute an important research area in the study of the spectrum of $\Delta$ and many works have been done during the past decades. Now we give a brief introduction of the main results.\\
     \indent
     The study of universal inequalities for (\ref{Dirichlet problem}) was initiated by Payne, P\"olya, and Weinberger \cite{PPW} in 1956, they proved
     \begin{eqnarray*}
     	\lambda_{k+1}-\lambda_{k} \leq \frac{4}{k n} \sum_{i=1}^{k} \lambda_{i}
     \end{eqnarray*}
     for $ \Omega \subset \mathbb{R}^{n}$ (they originally proved the case $n=2$, but their result can promoted to dimension $n$ exactly by the same proof  ).\\
     \indent
     In 1980, Hile and Protter \cite{HP} proved
     \begin{eqnarray*}
      \sum_{i=1}^{k} \frac{\lambda_{i}}{\lambda_{k+1}-\lambda_{i}} \geq \frac{k n}{4}.
     \end{eqnarray*}\\
     \indent
      In 1991, Yang \cite{Yang} (cf. \cite{CY1'}) proved a sharp universal inequality:
     \begin{eqnarray}\label{Yang's ineq.}
     	\sum_{i=1}^{k}\left(\lambda_{k+1}-\lambda_{i}\right)^2\leq\frac{4}{n}\sum_{i=1}^{k}\left(\lambda_{k+1}-\lambda_i\right) \lambda_{i},
     \end{eqnarray}
     which has been  called Yang's first inequality by Ashbaugh (cf. \cite{Ash1} and \cite{Ash2} and so on).

\indent
When $\Omega$ is a bounded domain on $\mathbb{S}^{n}(1)$, Cheng and Yang \cite{CY1} proved in 2005 that eigenvalues of  (\ref{Dirichlet problem}) satisfy
\begin{eqnarray}\label{Cheng's 2}
		\sum_{i=1}^{k}\left(\lambda_{k+1}-\lambda_{i}\right)^{2} \leq\frac{4}{n} \sum_{i=1}^{k}\left(\lambda_{k+1}-\lambda_{i}\right)\left(\lambda_i+\frac{n^2}{4}\right),
\end{eqnarray}
by canonically embedded $\mathbb{S}^{n}(1)$  into $\mathbb{R}^{n+1}$ (see also \cite{Chen} and \cite{Sou}). It is optimal since the above inequality becomes an equality for any $k$ when $\Omega=\mathbb{S}^n(1)$.

It is very natural to consider universal inequalities for eigenvalues of (\ref{Dirichlet problem}) when $M^n$ is the hyperbolic space $\mathbb{H}^{n}(-1)$. If $n=2$, by making use of estimates for eigenvalues of the Schr\"odinger like operator with a weight, Harrell and Michel \cite{HM} and Ashbaugh \cite{Ash2}
have obtained several results. For general $n$, Cheng and Yang \cite{CY2} proved in 2009 that eigenvalues of (\ref{Dirichlet problem}) satisfy
\begin{eqnarray}\label{Cheng's 1}
	\sum_{i=1}^{k}\left(\lambda_{k+1}-\lambda_{i}\right)^{2} \leqslant 4 \sum_{i=1}^{k}\left(\lambda_{k+1}-\lambda_{i}\right)\left(\lambda_{i}-\frac{(n-1)^{2}}{4}\right).
\end{eqnarray}

It remains a challenging problem if one can improve the coefficient on the right hand side of $(\ref{Cheng's 1})$ from $4$ to $\frac{4}{n}$ (which is sharp by making use of a recursion formula of Cheng and Yang \cite{CY1'} and Weyl's asymptotic formula (\ref{Weyl asy.})). Actually we have a conjecture raised by Cheng \cite{Che}.
\\
\\\textbf{Conjecture:} Assume that $\lambda_i$ are eigenvalues of  (\ref{Dirichlet problem}) with $(M^n, g)$ being the hyperbolic space $\mathbb{H}^{n}(-1)$. Then have we
\begin{eqnarray*}
	\sum_{i=1}^{k}\left(\lambda_{k+1}-\lambda_{i}\right)^{2} \leqslant \frac{4}{n} \sum_{i=1}^{k}\left(\lambda_{k+1}-\lambda_{i}\right)\left(\lambda_{i}-\frac{(n-1)^{2}}{4}\right).
\end{eqnarray*}	

Recently, motivated by Cheng's conjecture, Luo and Zheng \cite{LuoZ} studied eigenvalues of the Dirichlet Laplacian on the hyperbolic space $\mathbb{H}^{n}(-1)$, by viewing the hyperbolic space as a conformally flat Riemannian manifold. They proved 
\begin{thm}[\cite{LuoZ}]
		Assume that $(M^n,g )$ is the hyperbolic space $\mathbb{H}^{n}(-1)=(\mathbb{R}_{+}^{n}, \frac{d x_{1}^{2}+d x_{2}^{2}+\cdots+d x_{n}^{2}}{x_{n}^{2}})$, then eigenvalues $\lambda_{i}$ of (\ref{Dirichlet problem}) satisfy
		\begin{eqnarray}
			\sum_{i=1}^k(\lambda_{k+1}-\lambda_i)^2\leqslant \frac{\rho_{max}}{\rho_{min}}\frac{4}{n}\sum_{i=1}^k\left(\lambda_{k+1}-\lambda_i\right)\left(\lambda_i-\frac{n^{2}-2n-4}{4}\right),
		\end{eqnarray}
where $\rho_{max}:=\max_{\Omega}\frac{1}{x_n^2}$ and $\rho_{min}:=\min_\Omega\frac{1}{x_n^2}$.
	\end{thm}
In this paper, by choosing new test functions, we obtain new inequalities for eigenvalues of the Dirichlet Laplacian on the hyperbolic space. 
\begin{thm}\label{main thm1}
Assume that $(M^n,g)$ is the hyperbolic space $\mathbb{H}^{n}(-1)=(\mathbb{R}_{+}^{n}, \frac{d x_{1}^{2}+d x_{2}^{2}+\cdots+d x_{n}^{2}}{x_{n}^{2}})$, then eigenvalues $\lambda_{i}$ of (\ref{Dirichlet problem}) satisfy
		\begin{eqnarray}
			\sum_{i=1}^k(\lambda_{k+1}-\lambda_i)^2\leqslant \frac{\rho_{max}}{\rho_{min}}\frac{4}{n}\sum_{i=1}^k\left(\lambda_{k+1}-\lambda_i\right)\left(\lambda_i-\frac{(n-1)^2}{4}\right).
		\end{eqnarray}
	\end{thm}
Since $(n-1)^2>n^{2}-2n-4$, we see that Theorem \ref{main thm1} improves Theorem 1.1 due to Luo and Zheng. As a direct corollary, we verify Cheng's conjecture up to loss of $\epsilon$  for a special kind of bounded domains.
\begin{cor}
Assume that $\Omega$ is a bounded domain in the hyperbolic space $\mathbb{H}^{n}(-1)=(\mathbb{R}_{+}^{n}, \frac{d x_{1}^{2}+d x_{2}^{2}+\cdots+d x_{n}^{2}}{x_{n}^{2}})$ satisfying $$\frac{\rho_{max}}{\rho_{min}}\leq(1+\epsilon),$$ then eigenvalues $\lambda_{i}$ of (\ref{Dirichlet problem}) satisfy
		\begin{eqnarray}
			\sum_{i=1}^k(\lambda_{k+1}-\lambda_i)^2\leqslant \frac{4(1+\epsilon)}{n}\sum_{i=1}^k\left(\lambda_{k+1}-\lambda_i\right)\left(\lambda_i-\frac{(n-1)^2}{4}\right).
		\end{eqnarray}
\end{cor}
By choosing different test functions, we verify Cheng's conjecture up to loss of $\epsilon$ for another special kind of bounded domains.
\begin{thm}\label{main thm2}
Assume that $(M^n,g )$ is the hyperbolic space $\mathbb{H}^{n}(-1)=(\mathbb{R}_{+}^{n}, \frac{d x_{1}^{2}+d x_{2}^{2}+\cdots+d x_{n}^{2}}{x_{n}^{2}})$. Then for any $0<\epsilon\leq 1 (n=2)$ or $0<\epsilon\leq 2 (n\geq 3)$, if $\Omega$ is a bounded domain in $\mathbb{H}^n(-1)$ satisfying $$\max_{x\in\Omega}\sum_{p=1}^{n-1}\frac{(\frac{x_p}{x_n})^2}{1+(\frac{x_p}{x_n})^2}\leq\frac{\epsilon^2}{1+\epsilon},$$ then eigenvalues $\lambda_i$ of  (\ref{Dirichlet problem}) satisfy
	    $$\sum_{i=1}^{k}\left(\lambda_{k+1}-\lambda_{i}\right)^{2}\leq \frac{4(1+\epsilon)}{n}\sum_{i=1}^k(\lambda_{k+1}-\lambda_{i})(\lambda_i-\frac{(n-1)^2}{4}).$$
\end{thm}
\textbf{Organization.} In section 2 we collect some formulas from conformal geometry which will be useful in section 3 and in seciton 3 we give proofs of our main Theorem \ref{main thm1} and Theorem \ref{main thm2}.
\section{Preliminaries}
 Let $(M,g_0)$ be an $n$-dimensional Riemannian manifold. Given a function $h$$:$ $M$$ \rightarrow$ $\mathbb{R}$, we consider  the conformal manifold $(M, g)$ with the conformal metric $g=e^{2h}g_0$. Let $\nabla^0$, $\nabla$ denote the Riemannian connections and $\Delta^0$, $\Delta$ denote the Laplacian on  $(M, g_0)$ and $(M,g)$ respectively. 

 We denote $$\nabla^0 F\cdot\nabla^0 G=\langle \nabla^0 F, \nabla^0 G\rangle_{g_0}, \ \nabla F\cdot \nabla G=\langle \nabla F, \nabla G\rangle_g$$ for any functions $F,G$  $:$ $\Omega$ $\rightarrow$ $\mathbb{R}$. Then we have
\begin{align}\label{con0}
\nabla F \cdot\nabla G
&=\langle\nabla F, \nabla G\rangle_g\nonumber
	\\&=e^{2 h}\left\langle e^{-2 h} \nabla^0 F, e^{-2 h} \nabla^0 G\right\rangle_{g_0}\nonumber
	\\&=e^{-2 h}\langle\nabla^0 F, \nabla^0 G\rangle_{g_0}\nonumber
	\\&=e^{-2 h} \nabla^0 F\cdot \nabla^0 G.
\end{align}

  There is a known formula of the Laplacian under conformal transformation, which states that for a smooth function $F$ $:$ $M$ $\rightarrow$ $\mathbb{R}$,
	\begin{eqnarray}\label{con1}
		\quad \Delta F=e^{-2 h}\left(\Delta^0 F+(n-2) \nabla^0 h \cdot \nabla^0 F\right).
	\end{eqnarray}

\section{Proofs of Theorem \ref{main thm1} and Theorem \ref{main thm2}}
In the following we will use a general inequality for eigenvalues of the Dirichlet Laplacian due to Cheng and Yang \cite{CY}.
\begin{pro}[\cite{CY}]
Let $\lambda_i$ be the $i$-th eigenvalue of the above eigenvalue problem \ref{Dirichlet problem}
and $u_i$ be the orthonormal eigenfunction corresponding to $\lambda_i$, i.e.
\begin{eqnarray*}
\int_Mu_iu_i=\delta_{ij}, \int_M|\nabla u_i|^2=\lambda_i.
\end{eqnarray*}
 Then, for any function $f\in C^3(M)\cap C^2(\partial M)$ and any integer $k$, we have
\begin{eqnarray}\label{main ine}
\sum_{i=1}^k(\lambda_{k+1}-\lambda_i)^2\|u_i\nabla f\|^2\leq\sum_{i=1}^k(\lambda_{k+1}-\lambda_i)\|2\nabla f\cdot\nabla u_i+u_i\Delta f\|^2,
\end{eqnarray}
where $\|f\|^2=\int_Mf^2.$
\end{pro}
 Let $\mathbb{H}^n(-1)$ be the hyperbolic space with the Poincar$\acute{e}$ upper half space model with Riemannian metric $g$, i.e. $g=\frac{1}{x_n^2}(dx_1^2+\cdots+dx_n^2)$.
\vspace{0.5 cm}

\textbf{Proof of Theorem \ref{main thm1}.} Let $f_p=x_p, p=1,\cdots, n-1$, using formulas (\ref{con0}) and (\ref{con1}) we have $$|\nabla f_p|^2=x_n^2, \ \nabla f_p\cdot\nabla u_i=x_n^2\frac{\partial u_i}{\partial x_p}, \ and \ \Delta f_p=0.$$ Then we get 
$$\sum_{i=1}^k(\lambda_{k+1}-\lambda_i)^2\int_\Omega u_i^2x_n^2\leq4\sum_{i=1}^k(\lambda_{k+1}-\lambda_i)\int_\Omega x_n^4(\frac{\partial u_i}{\partial x_p})^2.$$
Therefore we obtain
$$\min_\Omega x_n^2\sum_{i=1}^k(\lambda_{k+1}-\lambda_i)^2\int_\Omega u_i^2\leq\max_\Omega x_n^24\sum_{i=1}^k(\lambda_{k+1}-\lambda_i)\int_\Omega x_n^2(\frac{\partial u_i}{\partial x_p})^2,$$
which implies 
\begin{eqnarray}\label{main ine2}
\sum_{i=1}^k(\lambda_{k+1}-\lambda_i)^2\leq\frac{\max_\Omega x_n^2}{\min_\Omega x_n^2}4\sum_{i=1}^k(\lambda_{k+1}-\lambda_i)\int_\Omega x_n^2(\frac{\partial u_i}{\partial x_p})^2,
\end{eqnarray}
for each $p=1,\cdots,n-1.$ 

Let $f_n=\ln x_n,$ using formulas (\ref{con0}) and (\ref{con1}) we have 
$$|\nabla f_n|^2=1,\ \nabla f_n\cdot\nabla u_i=x_n\frac{\partial u_i}{\partial x_n}, \ and \ \Delta f_n=1-n.$$ 
Then we get
$$\sum_{i=1}^k(\lambda_{k+1}-\lambda_i)^2\int_\Omega u_i^2\leq\sum_{i=1}^k(\lambda_{k+1}-\lambda_i)\int_\Omega\left(2x_n\frac{\partial u_i}{\partial x_n}+u_i(1-n)\right)^2.$$
Note that
\begin{eqnarray*}
\int_\Omega\left(2x_n\frac{\partial u_i}{\partial x_n}+u_i(1-n)\right)^2&=&\int_\Omega\left(4x_n^2(\frac{\partial u_i}{\partial x_n})^2+4(1-n)\nabla x_n\cdot\nabla u_i u_i+(n-1)^2u_i^2\right)
\\&=&\int_\Omega\left(4x_n^2(\frac{\partial u_i}{\partial x_n})^2+2(1-n)\nabla x_n\cdot\nabla u_i^2+(n-1)^2u_i^2\right)
\\&=&\int_\Omega\left(4x_n^2(\frac{\partial u_i}{\partial x_n})^2-2(1-n)\Delta x_n u_i^2+(n-1)^2u_i^2\right)
\\&=&\int_\Omega\left(4x_n^2(\frac{\partial u_i}{\partial x_n})^2-(n-1)^2u_i^2\right),
\end{eqnarray*}
we obtain
\begin{eqnarray}\label{main ine3}
&&\sum_{i=1}^k(\lambda_{k+1}-\lambda_i)^2\nonumber
\\&\leq&4\sum_{i=1}^k(\lambda_{k+1}-\lambda_i)\left(\int_\Omega x_n^2(\frac{\partial u_i}{\partial x_n})^2-\frac{(n-1)^2}{4}\right)\nonumber
\\&\leq&\frac{\max_\Omega x_n^2}{\min_\Omega x_n^2}4\sum_{i=1}^k(\lambda_{k+1}-\lambda_i)\left(\int_\Omega x_n^2(\frac{\partial u_i}{\partial x_n})^2-\frac{(n-1)^2}{4}\right).
\end{eqnarray}
Summing $p$ from $1$ to $n-1$ in inequality  (\ref{main ine2}) and inequality (\ref{main ine3}) we get
\begin{eqnarray*}
&&n\sum_{i=1}^k(\lambda_{k+1}-\lambda_i)^2
\\&\leq&\frac{\max_\Omega x_n^2}{\min_\Omega x_n^2}4\sum_{i=1}^k(\lambda_{k+1}-\lambda_i)\left(\int_\Omega|\nabla u_i|^2-\frac{(n-1)^2}{4}\right)
\\&=&\frac{\max_\Omega x_n^2}{\min_\Omega x_n^2}4\sum_{i=1}^k(\lambda_{k+1}-\lambda_i)\left(\lambda_i-\frac{(n-1)^2}{4}\right),
\end{eqnarray*}
where in the above inequality we used $|\nabla u_i|^2=x_n^2\left(\frac{\partial u_i}{\partial x_1})^2+\cdots+(\frac{\partial u_i}{\partial x_n})^2\right).$
Therefore we have
\begin{eqnarray*}
&&\sum_{i=1}^k(\lambda_{k+1}-\lambda_i)^2
\\&\leq&\frac{\max_\Omega x_n^2}{\min_\Omega x_n^2}\frac{4}{n}\sum_{i=1}^k(\lambda_{k+1}-\lambda_i)\left(\lambda_i-\frac{(n-1)^2}{4}\right)
\\&=&\frac{\rho_{max}}{\rho_{min}}\frac{4}{n}\sum_{i=1}^k(\lambda_{k+1}-\lambda_i)\left(\lambda_i-\frac{(n-1)^2}{4}\right),
\end{eqnarray*}
which completes the proof of Theorem \ref{main thm1}.

\vspace{0.5 cm}

\textbf{Proof of Theorem \ref{main thm2}.}  Choosing $f_p=arcsinh\frac{x_p}{x_n}, p=1,\cdots, n-1$, then using formulas (\ref{con0}) and (\ref{con1}) direct computations  show that 
\begin{eqnarray*}
\nabla f_p&=&x_n^2\nabla_0f_p
\\&=&x_n^2(\frac{\partial f_p}{\partial x_p}\frac{\partial}{\partial x_p}+\frac{\partial f_p}{\partial x_n}\frac{\partial}{\partial x_n})
\\&=&x_n\frac{1}{\sqrt{1+(\frac{x_p}{x_n}})^2}\frac{\partial}{\partial x_p}-x_n\frac{x_p}{x_n}\frac{1}{\sqrt{1+(\frac{x_p}{x_n}})^2}\frac{\partial}{\partial x_n},
\\\ |\nabla f_p|&=&1,  
\end{eqnarray*} 
and
\begin{eqnarray*}
\nabla f_p\cdot\nabla u_i&=&x_n^2\nabla^0f_p\cdot\nabla^0 u_i
\\&=&x_n^2(\frac{\partial f_p}{\partial x_p}\frac{\partial u_i}{\partial x_p}+\frac{\partial f_p}{\partial x_n}\frac{\partial u_i}{\partial x_n})
\\&=&\frac{1}{\sqrt{1+(\frac{x_p}{x_n})^2}}(x_n\frac{\partial u_i}{\partial x_p}-x_n\frac{x_p}{x_n}\frac{\partial u_i}{\partial x_n}),
\\  \Delta f_p&=&(n-1)\frac{sinh f_p}{cosh f_p}.
\end{eqnarray*}
Define $f_n=\ln x_n$, then using formulas (\ref{con0}) and (\ref{con1}) we have
$$|\nabla f_n|=1, \Delta f_n=1-n.$$

Denote by $u_{i,p}=\frac{\partial u_i}{\partial x_p}$ and $u_{i,n}=\frac{\partial u_i}{\partial x_n}$ in the sequel. Putting $f_p$ $(p=1,\cdots, n-1)$ into Cheng-Yang's inequality, we get
\begin{eqnarray}\label{main ine4}
&&\sum_{i=1}^k(\lambda_{k+1}-\lambda_i)^2\nonumber
\\&\leq&\sum_{i=1}^k(\lambda_{k+1}-\lambda_i)\int_\Omega[4(\nabla f_p\cdot\nabla u_i)^2+4\nabla f_p\cdot\nabla u_i u_i\Delta f_p+(\Delta f_p)^2]\nonumber
\\&=&\sum_{i=1}^k(\lambda_{k+1}-\lambda_i)\int_\Omega[4(\nabla f_p\cdot\nabla u_i)^2+2\nabla f_p\cdot\nabla u_i^2\Delta f_p+u_i^2(\Delta f_p)^2]\nonumber
\\&=&\sum_{i=1}^k(\lambda_{k+1}-\lambda_i)\int_\Omega[4(\nabla f_p\cdot\nabla u_i)^2-2u_i^2\nabla f_p\cdot\nabla\Delta f_p-u_i^2(\Delta f_p)^2]\nonumber
\\&=&\sum_{i=1}^k(\lambda_{k+1}-\lambda_i)[4\int_\Omega\frac{1}{1+{v_p^2}}(x_nu_{i,p}-v_px_nu_{i,n})^2\nonumber
\\&-&(n-1)^2\int_\Omega\frac{v_p^2}{1+v_p^2}u_i^2-2(n-1)\int_\Omega\frac{1}{1+v_p^2}u_i^2],
\end{eqnarray}
where  $v_p:=\frac{x_p}{x_n},$ and in the last inequality we used 
\begin{eqnarray*}
\nabla f_p\cdot\nabla\Delta f_p&=&(n-1)\nabla f_p\nabla\frac{\sinh f_p}{\cosh f_p}
\\&=&(n-1)\frac{1}{\cosh^2f_p}|\nabla f_p|^2
\\&=&(n-1)\frac{1}{1+v_p^2},
\\(\Delta f_p)^2&=&(n-1)^2\frac{v_p^2}{1+v_p^2}.
\end{eqnarray*}
Putting $f_n$ into Cheng-Yang's inequality, we get
$$\sum_{i=1}^k(\lambda_{k+1}-\lambda_i)^2\int_\Omega u_i^2\leq\sum_{i=1}^k(\lambda_{k+1}-\lambda_i)\int_\Omega\left(2x_nu_{i,n}+u_i(1-n)\right)^2.$$
Note that
\begin{eqnarray*}
\int_\Omega\left(2x_n\frac{\partial u_i}{\partial x_n}+u_i(1-n)\right)^2=4\int_\Omega\left(x_n^2u_{i,n}^2-u_i^2\frac{(n-1)^2}{4}\right),
\end{eqnarray*}
we obtain
\begin{eqnarray}\label{main ine5}
&&\sum_{i=1}^k(\lambda_{k+1}-\lambda_i)^2\nonumber
\\&\leq&4\sum_{i=1}^k(\lambda_{k+1}-\lambda_i)\left(\int_\Omega x_n^2u_{i,n}^2-\frac{(n-1)^2}{4}\right).
\end{eqnarray}
 Summing over $p=1,...,n-1$ in (\ref{main ine4}) and (\ref{main ine5}), we get
\begin{eqnarray*}
&&n\sum_{i=1}^k(\lambda_{k+1}-\lambda_i)^2
\\&\leq&\sum_{i=1}^k(\lambda_{k+1}-\lambda_i)[4\sum_{p=1}^{n-1}\int_\Omega\frac{1}{1+{v_p^2}}
(x_nu_{i,p}-v_px_nu_{i,n})^2+4\int_\Omega x_n^2u_{i,n}^2
\\&-&(n-1)^2-(n-1)^2\sum_{p=1}^{n-1}\int_\Omega\frac{v_p^2}{1+v_p^2}u_i^2-2(n-1)\sum_{p=1}^{n-1}\int_\Omega\frac{1}{1+v_p^2}u_i^2].
\end{eqnarray*}
 Now by Cauchy-Schwarz inequality, for any $\epsilon>0$
\begin{eqnarray*}
&&\sum_{p=1}^{n-1}\frac{1}{1+v_p^2}
(x_nu_{i,p}-v_px_nu_{i,n})^2
\\&\leq&\sum_{p=1}^{n-1}\frac{1+\epsilon}{1+v_p^2}x_n^2u_{i,p}^2+(1+\frac{1}{\epsilon})\sum_{p=1}^{n-1}\frac{v_p^2}{1+v_p^2}x_n^2u_{i,n}^2
\\&\leq&(1+\epsilon)\sum_{p=1}^{n-1}x_n^2u_{i,p}^2+(1+\frac{1}{\epsilon})\sum_{p=1}^{n-1}\frac{v_p^2}{1+v_p^2}x_n^2u_{i,n}^2
\\&\leq&(1+\epsilon)\sum_{p=1}^{n-1}x_n^2u_{i,p}^2+\epsilon x_n^2u_{i,n}^2,
\end{eqnarray*}
where in the last inequality we used the assumption 
$\max_{x\in\Omega}\sum_{p=1}^{n-1}\frac{(\frac{x_p}{x_n})^2}{1+(\frac{x_p}{x_n})^2}\leq\frac{\epsilon^2}{1+\epsilon}.$
Therefore \begin{eqnarray*}
&&n\sum_{i=1}^k(\lambda_{k+1}-\lambda_i)^2
\\&\leq&\sum_{i=1}^k(\lambda_{k+1}-\lambda_i)[4(1+\epsilon)\sum_{p=1}^{n-1}\int_\Omega
x_n^2u_{i,p}^2+4(1+\epsilon)\int_\Omega x_n^2u_{i,n}^2
\\&-&(n-1)^2-(n-1)^2\sum_{p=1}^{n-1}\int_\Omega\frac{v_p^2}{1+v_p^2}u_i^2-2(n-1)\sum_{p=1}^{n-1}\int_\Omega\frac{1}{1+v_p^2}u_i^2].
\end{eqnarray*}
Note that $\sum_{p=1}^{n-1}x_n^2u_{i,p}^2+x_n^2u_{i,n}^2=|\nabla u_i|^2$. Thus
\begin{eqnarray*}
&&\sum_{i=1}^k(\lambda_{k+1}-\lambda_i)^2
\\&\leq&\frac{4(1+\epsilon)}{n}\sum_{i=1}^k(\lambda_{k+1}-\lambda_i)\{\lambda_i-
\\&\frac{1}{4(1+\epsilon)}&[(n-1)^2+(n-1)^2\sum_{p=1}^{n-1}\int_\Omega\frac{v_p^2}{1+v_p^2}u_i^2+2(n-1)\sum_{p=1}^{n-1}\int_\Omega\frac{1}{1+v_p^2}u_i^2]\}.
\end{eqnarray*}
In the end, it is easy to check that if $0<\epsilon\leq1 (n=2)$ or $0<\epsilon\leq 2 (n\geq3)$, we have 
\begin{eqnarray*}
\frac{1}{4(1+\epsilon)}[(n-1)^2+(n-1)^2\sum_{p=1}^{n-1}\int_\Omega\frac{v_p^2}{1+v_p^2}u_i^2+2(n-1)\sum_{p=1}^{n-1}\int_\Omega\frac{1}{1+v_p^2}u_i^2]\geq\frac{(n-1)^2}{4}.
\end{eqnarray*}
This completes the proof of Theorem \ref{main thm2}.

\vspace{0.5 cm}

\textbf{Declarations.} 
\\
\\ \textbf{Data availability statements.} No data sets were generated or analysed during
the current study.

\vspace{0.5 cm}
\textbf{Acknowledgement.} This paper is supported by the National Natural Science Foundation of China (Grant nos. 12271069, 12571055). The author thank Professor Qingming Cheng for stimulating discussions on this subject. Part of this work was carried out when the author visited Xinyang Normal University. The author would like to thank Professor Yingbo Han and Professor Jiabin Yin for their invitation, stimulating discussions and Xinyang University for the hospitality. 
	{}
\vspace{1cm}\sc
	
Yong Luo

Mathematical Science Research Center of Mathematics,

Chongqing University of Technology,

Chongqing, 400054, China

{\tt yongluo-math@cqut.edu.cn}

\end{document}